\documentclass[11pt]{article}
\usepackage{latexsym, amsmath, amssymb, amsfonts, amscd}
\usepackage{graphics}
\usepackage[dvips]{epsfig}
\usepackage[pagebackref]{hyperref}
\usepackage{amsthm}
\usepackage[all]{xy}
\usepackage{stmaryrd}
\usepackage{bbm}
\usepackage{authblk}
\usepackage{tikz-cd}
\newtheorem{thm}{Theorem}[section]  
\newtheorem{prop}[thm]{Proposition}
\newtheorem{lemma}[thm]{Lemma}

\newtheorem{dfn}[thm]{Definition}

\usepackage[utf8]{inputenc}
\newtheorem{rmk}[thm]{Remark}
\newtheorem{ex}[thm]{Example}

\newcommand{\vbgs}[4]{
\xymatrix{#1 \ar[d]_{\tilde p} \ar@<2pt>[r] \ar@<-2pt>[r] & #2 \ar[d]^p \\ #3 \ar@<2pt>[r] \ar@<-2pt>[r] & #4}}
\begin{document}

\newcommand{\vbal}[4]{
\xymatrix{#1 \ar[d]_{\tilde p} \ar[r] & #2 \ar[d]^p \\ #3 \ar[r]  & #4}}
\newcommand\lie{\mathfrak}

\title{\bf Deformations of pairs of codimension one foliations }

\author{ Ameth Ndiaye \thanks{ Address:
    D\'epartement de Mathematiques, FASTEF, Universit\'e Cheikh Anta Diop, Dakar,  S\'en\'egal Email: ameth1.ndiaye@ucad.edu.sn} \
 and   A\"issa Wade \thanks{Address: Department of Mathematics,
 The Pennsylvania State University,
University Park, PA 16802, U.S.A., and University of Waterloo, Ontario, Canada.  \qquad \qquad \qquad 
 Email: wade@math.psu.edu}}

\date{}

\maketitle

\begin{abstract}
The notion  of a linear deformation of  a codimension one foliation into contact structures was introduced in \cite{Dathe-Rukimbira}. This concept  is a special type of deformation of confoliations.
 In this paper, we study  linear deformations of pairs of codimension one foliations into contact pairs. Applications  of our main result are also provided. \end{abstract}
    \noindent {\bf Keywords:} {\small Codimension one foliation, linear deformation, contact pairs.}

\noindent {\bf Mathematics Subject Classification 2010}:57C15, 53C57.


 \section{Introduction}
Confoliations interpolate between  codimension one foliations and contact structures. Their study was initiated by Eliashberg and  Thurston (see \cite{Eliashberg})  who investigated  their deformations. Later on, Etnyre pointed out (see \cite{EP16}) that every contact structure on a closed oriented 3-manifold can be obtained from deformations of foliations. This result was extended to higher dimensions on  some Lie groups by the first author. In this paper, we study   deformations of pairs of confoliations. More precisely, we study the  question:\\

\noindent {\bf Problem:}  Let $M$ be an  even-dimensional manifold. Given a  pair $(\mathcal{F}^1, \mathcal{F}^2)$ of  codimension 1 foliations on $M$  defined by a pair $(\alpha_0, \beta_0)$  of closed non-singular 1-forms,  under which  conditions  the pair $(\mathcal{F}^1, \mathcal{F}^2)$    can be  linearly deformed into a  pair $(\mathcal{F}^1_t, \mathcal{F}^2_t)$  of characteristic foliations of pairs of 1-forms  $(\alpha_t, \beta_t)$ which are transverse and complementary? \\

 We observe that  $(\alpha_t, \beta_t)$ could be  a contact pair in the sense of Bande and Hadar \cite{Bande3}. By a contact pair of type $( k, \ell)$ on $M$, we mean  a pair $(\alpha, \beta)$ of Pfaffian forms  such that
 $\alpha\wedge  d \alpha^k \wedge \beta \wedge d\beta^\ell$ is a volume form on $M$, $d \alpha^{k+1}=0$, and  $d\beta^{\ell +1}=0.$
 In this case,  the characteristic foliations of $\alpha$ and $\beta$  are  transverse and complementary and their leaves  are contact submanifolds  coming from regular Jacobi structures on $M$. Indeed, we know that given a regular Jacobi structure on $M$, there is a regular foliation defined by the distribution of Hamiltonian vector fields whose leaves are either contact submanifolds or locally conformal symplectic submanifolds.   It turns out that  these  two  characteristic foliations are contact foliations of Jacobi structures on $M$.   \\
 
  The paper is divided into  six short sections. In Section 2, we review deformations of confoliations via a linear deformation of their defining 1-form. Section 3 is  devoted to contact fiber bundles in the sense of Lerman \cite{Lerman}.  In Section 4, we recall basic results on contact pairs and Jacobi structures that will be used later.  In Section 5,   we examine  linear  deformations of pairs of foliations. We prove that a necessary and sufficient condition for such a  linear  deformation to exist is that the manifold  admits a contact pair (see Theorem \ref{thm1}).
 As  applications of Theorem \ref{thm1},  we  construct examples on Lie groups and  contact fiber bundles. 

\section{Deformations of confoliations}

All manifolds considered here are  smooth, connected, and orientable. 
Recall that a  \emph{contact form }on a manifold $M$ of dimension $2n+1$ is a 1-form $\alpha$ on $M$ such that $\alpha\wedge(d\alpha)^n\neq 0$ at any point $m \in M$. The condition $\alpha\wedge(d\alpha)^n\neq 0$ everywhere in $M$ means that  the distribution $\xi={\rm ker}(\alpha)$  corresponds to  a  maximally non-integrable  hyperplane field on $M$. 
After fixing an orientation on $M$  we  can  distinguish between the positive and the  negative contact condition, i.e.
$$\star (\alpha\wedge(d\alpha)^n) >0 \quad{\rm or} \quad \star (\alpha\wedge(d\alpha)^n) <0.$$  By a \emph{ positive (resp. negative) confoliation}, we mean that 

$$\star (\alpha\wedge(d\alpha)^n) \geq 0 \quad ({\rm resp.} ~ \star (\alpha\wedge(d\alpha)^n) \leq 0 ).$$

\noindent In their work  on deformations of confoliations \cite{Eliashberg},  Eliashberg and Thurston  proved   that any oriented   $C^2$-foliation  of codimension 1 on an oriented $3$-manifold $M$ can be approximated by contact structures, except when it is the product foliation on  $\mathbb{S}^2\times\mathbb{S}^1$. 
\begin{dfn}
 A foliation $\mathcal{F}$ of codimension 1  on $M$ is called deformable into contact structures if there exist a one parameter family $\mathcal{F}_t$ of hyperplane fields satisfying $\mathcal{F}_0=\mathcal{F}$ and  $\mathcal{F}_t$ is  a contact distribution, for all $t>0.$  
\end{dfn}

Now, suppose  $\mathcal{F}$  is defined by a  closed 1-form  $\alpha_0.$  Following \cite{Dathe-Rukimbira}, we say that a deformation $\mathcal{F}_t$ of  $\mathcal{F}$  is \emph{linear} if there is a 1-parameter family of 1-forms of the form: $\alpha_t=\alpha_0+t\alpha$, where $\alpha$ is a fixed 1-form on $M$ with $\mathcal{F}_t={\rm Ker}(\alpha_t)$. 
The following result was proven  by Dathe and Rukimbira:
\begin{thm} \cite {Dathe-Rukimbira}
Let $M$ be a closed $(2n+1)$-dimensional manifold, let  $\alpha_0$ be  a closed 1-form and $\alpha$ a 1-form on $M$. Then, the following conditions are equivalent.
\begin{itemize}
    \item[(i)] The linear deformation  $\alpha_t=\alpha_0+t\alpha$ of $\alpha_0$ is a  contact  1-form for all $t>0$.
    \item[(ii)] The 1-form $\alpha$ is contact and $\alpha_0(Z)=0,$ where $Z$ is the Reeb vector field of $\alpha$.
\end{itemize}
\end{thm}
In the special case where $M=G$ is a K-contact Lie group, the first author proved the following result:
\begin{thm}\cite{Ameth} \label{Ameth}
Let $\mathcal{F}$ be a foliation defined by a nowhere vanishing  right invariant 1-form $\alpha^+_0$ in a
K-contact Lie group G of dimension $2n + 1$. If $\alpha^+_0$ is harmonic with respect to some contact metric $h$ then  $\mathcal{F}$ can be  deformed into  contact structures.
\end{thm}

\section{ Contact fiber bundles}

The existence of a contact structure on a given  odd-dimensional manifold has been a fundamental question in  contact geometry which lead to construction methods developed by Geiges \cite{Geiges97}.  This existence question was  recently solved by  
Borman, Eliashberg, and Murphy \cite{BEM15}.  They  proved that contact structures exist on odd-dimensional manifolds whenever almost contact structures exist. 
The following theorem which will be used later in this paper, was proven by  Bourgeois in \cite{Bourgeois}. It provides another construction method for contact structures. 

\begin{thm} \cite{Bourgeois}\label{Theorem Bourgeois}
Let $N^{2n-3}$ be a closed contact manifold of dimension $2n -3$, $n\geq 3$. Then  $M=\mathbb T^2 \times N^{2n-3}$  admits a contact structure that is invariant under the natural $\mathbb T^2$-action, and such that $\lbrace p \rbrace \times N$ is a contact submanifold, for every $p\in \mathbb T^2$.
\end{thm}

In  Bourgeois's construction,  the projection $p: \mathbb T^2 \times N^{2n-3} \to \mathbb T^2$ can be considered as a contact fiber bundle in the sense of Lerman \cite{Lerman}. 
\begin{dfn} \cite{Lerman}
Let $(F, \eta_F)$ be a contact manifold. A fiber bundle $\pi: M\to B$ with typical fiber $F$ is called a contact fiber bundle if there is an open cover 
$(U_{i})_{i\in I}$ of $B$ with  local trivializations  $\phi_i: \pi^{-1}(U_i)\to U_i \times F$  such that for each $b \in U_i\cap U_j$,  the transition map 
$(\phi_j \circ \phi_i^{-1})_{|\{b\} \times F}$ is a  contactomorphism of $F$. 

\end{dfn}

\noindent Contact fiber bundles provide  examples of construction of higher-dimensional contact manifolds with very interesting properties (see \cite{Gironella2020, EP16}) and References therein). 
In fact, the total space $M$  of a contact fiber bundle $F \hookrightarrow M \to B$  is naturally equipped with a Jacobi structure  (see below) coming from the contact structure of each fiber. 

\section{Contact pairs and Jacobi structures}

By the characteristic subspace of a differential 1-form $\alpha \in \Omega^1(M)$ at  $m \in M$, we mean the subspace $C_m \subset T_mM$ that consists of the intersection of  ${\rm ker}(\alpha_m)$ and ${\rm ker}(d\alpha_m)$.
The  class of $\alpha$ at $m$ is $2k+1$ if  
\begin{eqnarray}
\alpha \wedge (d \alpha)^k \ne 0 \quad {\rm and} \quad  (d\alpha)^{k+1}=0.
\end{eqnarray} 
Clearly, if $\alpha$  has  constant class  then its characteristic distribution is completely integrable and it gives rise to  a regular foliation called the  characteristic foliation of $\alpha$. 
So,  1-forms of class zero are  exactly closed non-singular  1-forms while contact forms on an odd-dimensional manifold $M$ are  1-forms on $M$ of maximal class.
 Moreover,  contact forms on $M$ are 1-forms whose characteristic foliation consists of one single leaf.\\

\noindent Let $M$ be an even-dimensional manifold.   Any pair $(\alpha, \beta)$ of  1-forms of  class  $2k+1$ and $2\ell+1$, respectively for which
$ \alpha \wedge (d \alpha)^k \wedge  \beta \wedge (d \beta)^\ell$ is  a volume form  on $M$ will be called  \emph{contact pair  on $M$ of type}  $(k,\ell)$ following \cite{Bande3}. 

\medskip
\noindent The following result was proved in \cite{Bande3}:
\begin{lemma}
 To any  contact pair $(\alpha, \beta)$ on $M$,  it corresponds a pair $(E_\alpha, E_\beta)$ of commuting vector fields uniquely determined by the relations:
 $$\alpha(E_\alpha)=\beta(E_\beta)=1, \quad  \alpha(E_\beta)=\beta(E_\alpha)=0,$$ and $$ i_{E_\alpha}d \alpha=i_{E_\beta}d \alpha=i_{E_\alpha}d \beta= i_{E_\beta}d \beta=0.$$
\end{lemma}

\noindent The pair  $(E_\alpha, E_\beta)$  is called the \emph{pair of Reeb vector fields} for  $(\alpha, \beta)$.

 \begin{ex}\cite{Bande3} \label{Example C.P}
 Let $(M^{2k+1}, \alpha)$ and $(N^{2\ell+1}, \beta)$ be two contact manifolds and $P = M^{2k+1}\times N^{2\ell+1}$. Then  $(\alpha, \beta)$ is a contact pair of type $(k, \ell)$ on $P,$ called  a \emph{product contact pair}. Its Reeb fields are those of the two contact forms considered as vector
fields on $P$.
 \end{ex}
 
 \begin{rmk}\label{Darboux}
 The above example is exactly the local model for contact pair. In fact  Theorem 3.1 from \cite{Bande3} states that any point $m$ of a  smooth manifold $M$ endowed with a contact pair has a  \emph{Darboux neighborhood}  that is the product of  two contact manifolds $(U_1, \alpha)$ and $(U_2, \beta)$. 
 \end{rmk}
 
\noindent  Obviously,  any manifold equipped with a contact pair  $(\alpha, \beta)$  is orientable. Moreover,  the characteristic foliations of $\alpha$ and $\beta$  are  two  transverse and complementary  foliations whose leaves  are contact manifolds of dimension  $2k+1$ and $2\ell+1,$ respectively. 
Notice that  any closed 2-dimensional manifold equipped with a contact pair of type (0,0), that is a pair of closed 1-forms $(\alpha, \beta)$ with $\alpha \wedge \beta \ne 0$, is diffeomorphic to the torus $\mathbb T^2$.  We are interested in the case where $k \geq 1$ or $\ell \geq 1$. In the sequel, we assume $k \geq 1$.\\

Contact pairs are  closely related to Jacobi structures. Recall that a Jacobi structure on a manifold $M$  in the sense of Lichnerowicz \cite{Lichne} ( see \cite{Vitagliano-Wade2} and References therein) is defined by a bivector  field $\Lambda$ and a vector field $E$ such that 
\begin{eqnarray}
[\Lambda, \Lambda]= 2 E\wedge \Lambda \quad {\rm and} \quad  [\Lambda, E]=0.
\end{eqnarray}
There is an  associated Jacobi bracket which is  a Lie bracket on $C^\infty(M)$ defined by:
\begin{eqnarray}
\{f, g\}= \Lambda(df, dg)+ f (E\cdot g)  - g(E \cdot f),\quad \forall~  f, g \in C^\infty(M).
\end{eqnarray}
This Lie bracket is local in the sense of Kirillov \cite{Kirillov}, i.e. 
\begin{eqnarray}
{\rm supp} (\{f, g\}) \subset {\rm supp} (f) \cap {\rm supp} (g).
\end{eqnarray}
Any contact pair  $(\alpha, \beta)$ on $M^{2n}$ determines a pair of Jacobi structures on $M$  whose characteristic foliations are transverse and complementary. Indeed,  the contact structures on leaves of the characteristic foliation $\mathcal{F}$ (resp.  $\mathcal{G}$) of $\alpha$  (resp. $\beta$)  determines a Jacobi bracket defined as follows:  to each $f \in C^\infty(M),$ it  is associated a vector field  $X_f^\alpha$  (resp. $X_f^\beta$) tangent to leaves of $\mathcal{F}$  (resp.  $\mathcal{G}$).   Furthermore, if  $L$ is a leaf of the characteristic foliation of $\alpha$ then the vector field  $X_f^\alpha$ satisfies:
\begin{eqnarray}
\alpha_L (X_f^\alpha)= f_{|_{L}}  \quad {\rm and}\quad( L_{X_f^\alpha} \alpha_L) \wedge \alpha_L=0,
\end{eqnarray}
where $\alpha_L$ is the  contact form induced by $\alpha$ on the leaf $L$. A similar expression for $X_f^\beta$ can be written.  Moreover,  the Jacobi brackets are defined as follows:
\begin{eqnarray}
\{f, g\}_{\alpha}= \alpha ([X_f^\alpha, X_g^\alpha ]) \quad \{f, g\}_{\beta}= \beta ([X_f^\beta, X_g^\beta ]).
\end{eqnarray}
In particular, we call $E^\alpha$  the  unique vector field defined by:

\begin{eqnarray}
\alpha_L (E^\alpha)= 1_{|_{L}},
\end{eqnarray}
where  $1_{|_{L}}$ is the restriction of the constant function $f=1$ on the leaf $L$. We have a similar expression for $E^\beta$. The bivectors associated with the Jacobi brackets are given by:
\begin{eqnarray}
\Lambda^\alpha(df, dg)= \{f, g\}_{\alpha}- f (E^{^\alpha}\cdot g)+ g (E^{^\alpha} \cdot f)
\end{eqnarray}
 and 
 \begin{eqnarray}
 \Lambda^\beta(df, dg)= \{f, g\}_{\beta}- f (E^{^\beta} \cdot g)+ g (E^{^\beta}\cdot f).
 \end{eqnarray}  
 Hence $(\Lambda^\alpha, E^\alpha)$ and  $(\Lambda^\beta, E^\beta)$ are two Jacobi structures  on $M$ whose characteristic foliations are exactly  those of $\alpha$ and $\beta$.

\section{Deformation of pairs of   codimension 1 foliations}
Codimension one  foliations considered here are those  defined by a  nowhere vanishing  differential 1-form. Tischler' s theorem \cite{Tischler} states that  any compact manifold that has a closed nowhere vanishing 
1-form is a fibration over a circle.
As mentioned above, this paper is devoted to the study of simultaneous deformations of  pairs of confoliations. It's known that  manifolds that admit such a  pair of foliations have  interesting properties. In particular, we have:

\begin{thm} \cite{Cardona-Miranda}
Suppose  $M$ is  a  compact connected manifold which admits  two  closed 1-forms $\alpha_0$ and  $\beta_0$ that are linearly independent at every point of $M$.
Then $M$ is a fibration over $\mathbb T^2$.
\end{thm}

\begin{dfn}
Let $\alpha_0$ and  $\beta_0$ be two  closed 1-forms on $M$   that are linearly independent at every point $m\in M$.
A linear deformation of $(\alpha_0, \beta_0)$ is a  pair $\left(\alpha_t, \beta_t\right)$ of  1-forms of the type:
\begin{eqnarray} \label{linear deformation}
 \alpha_t=\alpha_0+t\alpha \quad {\rm and} \quad  \beta_t=\beta_0+t\beta, ~ ~ ~ t \in \mathbb R,
 \end{eqnarray}  
 for some  $\alpha$ and $\beta \in \Omega^1(M)$.
\end{dfn}

\noindent We obtain the following result:
\begin{thm}\label{thm1}
Let $M$ be a closed  and oriented even-dimensional manifold  and let $\mathcal{F}^1$ and $ \mathcal{F}^2$ be  two foliations in  $M$  defined  by two  linearly independent closed 1-forms $\alpha_0$ and  $\beta_0$, respectively. Consider the linear deformation $(\alpha_t, \beta_t)$ given by:  
 \begin{eqnarray} \label{linear deformation}
 \alpha_t=\alpha_0+t\alpha \quad {\rm and} \quad  \beta_t=\beta_0+t\beta, ~ \forall ~ t \in \mathbb R,
 \end{eqnarray}  
 \noindent where $\alpha$ and $\beta$ are two  1-forms on $M.$
Then  $\left(\alpha_t, \beta_t\right)$  is  a contact pair of type $(k, \ell)$  for $t > 0$   whose associated pair of Reeb vector fields has the form $( { 1 \over t} X,  { 1 \over t} Y)$
(where  $X$  $Y$  are vector fields on $M$)   if only if $\left(\alpha, \beta\right)$  is a  contact pair of type $(k, \ell)$ whose pair of  Reeb vector fields 
$(E_\alpha, E_\beta)$ satisfies the compatibility conditions $ \alpha_0(E_\alpha)=\alpha_0(E_\beta)= \beta_0(E_\alpha)=\beta_0(E_\beta)=0.$ 
\end{thm}

The proof of  Theorem \ref{thm1} will be divided in two parts. First, we will establish the proposition:

\begin{prop}\label{prop1}
Let $M$ be a  $2n$-dimensional manifold. Let $\mathcal{F}^1$ and $ \mathcal{F}^2$ be  two foliations in  $M$  defined  by two  linearly independent closed non-singular 1-forms $\alpha_0$ and  $\beta_0$, respectively. Consider the linear deformation $(\alpha_t, \beta_t)$ given by:  
 \begin{eqnarray} 
 \alpha_t=\alpha_0+t\alpha \quad {\rm and} \quad  \beta_t=\beta_0+t\beta, ~ \forall ~ t \in \mathbb R,
 \end{eqnarray}  
 where $\alpha$ and $\beta$ are two  1-forms on $M.$ Suppose $\left(\alpha, \beta\right)$  is a  contact pair of type $(k, \ell)$  whose associated pair of  Reeb vector fields 
$(E_\alpha, E_\beta)$ satisfies the compatibility relations  $ \alpha_0(E_\alpha)=\alpha_0(E_\beta)=\beta_0(E_\beta)=\beta(E_\alpha)=0$.
Then  $\left(\alpha_t, \beta_t \right)$  is  a contact pair of type $(k, \ell)$ for $t > 0.$ Moreover,  its associated pair of Reeb vector fields is $( { 1 \over t} E_\alpha,  { 1 \over t} E_\beta).$ 
\end{prop}

To prove  Proposition \ref{prop1}, we use the following  two lemmas:

\begin{lemma}\label{lemma5.3}
Let $M$ be an  even-dimensional manifold equipped with a  contact pair $\left(\alpha, \beta\right)$ of type $(k, \ell).$ Denote by $(E_\alpha, E_\beta)$ its associated pair of  Reeb vector fields.
 For any 1-forms $\omega$  and $\overline{\omega}$ on $M$,we have:
\begin{equation}\label{P1} \omega \wedge\left(d\alpha\right)^k\wedge\beta\wedge\left(d\beta\right)^\ell\  = \omega(E_\alpha) \alpha \wedge\left(d\alpha\right)^k\wedge\beta\wedge\left(d\beta\right)^\ell; \ \end{equation}
\begin{equation}\label{P2} \omega \wedge \alpha \wedge  \left(d\alpha\right)^k\wedge\left(d\beta\right)^\ell\  = -\omega(E_\beta) \alpha \wedge\left(d\alpha\right)^k\wedge\beta\wedge\left(d\beta\right)^\ell.\ \end{equation}
\end{lemma}

\begin{proof}
 Since the dimension of $M$ equals $2k+2l+2$, we have:
 $$\omega\wedge\alpha\wedge\left(d\alpha\right)^k\wedge\beta\wedge\left(d\beta\right)^\ell\\=0.$$ Take the contraction  by $E_\alpha$ to obtain:
 $$i_{E_\alpha}\left(\omega\wedge\alpha\wedge\left(d\alpha\right)^k\wedge\beta\wedge\left(d\beta\right)^\ell\\\right)=0.$$ That is, 
 \begin{eqnarray*}
 0&=& \omega(E_\alpha)\alpha\wedge\left(d\alpha\right)^k\wedge\beta\wedge\left(d\beta\right)^\ell 
 - \omega\wedge\mathfrak{i}_{E_\alpha}\left(\alpha\wedge\left(d\alpha\right)^k\wedge\beta\wedge\left(d\beta\right)^\ell \right) \\
 &=&\omega(E_\alpha)\alpha\wedge\left(d\alpha\right)^k\wedge\beta\wedge\left(d\beta\right)^\ell -\omega\wedge\alpha(E_\alpha)\left(d\alpha\right)^k\wedge\beta\wedge\left(d\beta\right)^\ell \\
 && +\omega\wedge\alpha\wedge i_{E_\alpha}\left(\left(d\alpha\right)^k\wedge\beta\wedge\left(d\beta\right)^\ell \right). \end{eqnarray*} But $\alpha(E_\alpha)=1$ and $ i_{E_\alpha}\left(\left(d\alpha\right)^k\wedge\beta\wedge\left(d\beta\right)^\ell \right)=0$ therefore
 $$\omega \wedge\left(d\alpha\right)^k\wedge\beta\wedge\left(d\beta\right)^\ell\  = \omega(E_\alpha) \alpha \wedge\left(d\alpha\right)^k\wedge\beta\wedge\left(d\beta\right)^\ell.$$ By similar arguments, using $\beta$ and  contraction with $E_\beta$,  we get
 $$\omega \wedge \alpha \wedge  \left(d\alpha\right)^k\wedge\left(d\beta\right)^\ell\  = -\omega(E_\beta) \alpha \wedge\left(d\alpha\right)^k\wedge\beta\wedge\left(d\beta\right)^\ell.$$
\end{proof}

\begin{lemma}\label{lemma5.4}
Let $M$ be an  even-dimensional manifold equipped with a  contact pair $\left(\alpha, \beta\right)$ of type $(k, \ell).$ Denote by $(E_\alpha, E_\beta)$ its associated pair of  Reeb vector fields.
 For any 1-forms $\omega$  and $\overline{\omega}$ on $M$ satisfying $\omega(E_\beta)=\overline{\omega}(E_\beta)=0$, we have:
\begin{equation}\label{P3} \omega\ \wedge\left(d\alpha\right)^k  \wedge  \overline{\omega} \wedge\left(d\beta\right)^\ell =  0.\end{equation}
\end{lemma}
\begin{proof} The proof is similar to that of the above lemma. We have:
$$\beta\wedge\omega\ \wedge\left(d\alpha\right)^k  \wedge  \overline{\omega} \wedge\left(d\beta\right)^\ell=  0$$  since the dimension of $M$ equals $2k+2l+2$. It follows
\begin{eqnarray*}
0&=& i_{E_\beta}\left(\beta\wedge\omega\ \wedge\left(d\alpha\right)^k  \wedge  \overline{\omega} \wedge\left(d\beta\right)^\ell\right) \\
&=&\beta(E_\beta)\omega\ \wedge\left(d\alpha\right)^k  \wedge  \overline{\omega} \wedge\left(d\beta\right)^\ell -\beta\wedge\textit{i}_{E_\beta}\left(\omega\ \wedge\left(d\alpha\right)^k  \wedge  \overline{\omega} \wedge\left(d\beta\right)^\ell\right) \\
&=& \omega\ \wedge\left(d\alpha\right)^k  \wedge  \overline{\omega} \wedge\left(d\beta\right)^\ell -\beta\wedge\left(\omega(E_\beta)\left(d\alpha\right)^k  \wedge  \overline{\omega} \wedge\left(d\beta\right)^\ell \right) \\
&&-\beta\wedge\left(  \omega\wedge i_{E_\beta}\left(\left(d\alpha\right)^k  \wedge  \overline{\omega} \wedge\left(d\beta\right)^\ell\right)\right)\\
&=&  \omega\ \wedge\left(d\alpha\right)^k  \wedge  \overline{\omega} \wedge\left(d\beta\right)^\ell -\omega(E_\beta) \beta\wedge \left(d\alpha\right)^k  \wedge  \overline{\omega} \wedge\left(d\beta\right)^\ell 
\end{eqnarray*}
since $0=\overline{\omega}(E_\beta)=i_{E_\beta}\left(d\alpha\right)^k=i_{E_\beta}\left(d\alpha\right)^\ell.$ Using $\omega(E_\beta)=0$, we get $$ \omega\ \wedge\left(d\alpha\right)^k  \wedge  \overline{\omega} \wedge\left(d\beta\right)^\ell=0.$$ This completes the proof.
\end{proof}

\noindent Now, we will proceed with the proof of   Proposition \ref{prop1}.
\begin{proof}
Under the assumptions in  Proposition \ref{prop1}, we have:
\begin{eqnarray}\label{1}
 \alpha_t\wedge\left(d\alpha_t\right)^k\wedge\beta_t\wedge\left(d\beta_t\right)^\ell = t^{k+\ell}\alpha_0\wedge\left(d\alpha\right)^k\wedge\beta_0\wedge\left(d\beta\right)^\ell\nonumber\\
 +t^{k+\ell+1}\alpha_0\wedge\left(d\alpha\right)^k\wedge\beta\wedge\left(d\beta\right)^\ell\nonumber\\
 +t^{k+\ell+1}\alpha\wedge\left(d\alpha\right)^k\wedge\beta_0\wedge\left(d\beta\right)^\ell\nonumber\\
 +t^{k+\ell+2}\alpha\wedge\left(d\alpha\right)^k\wedge\beta\wedge\left(d\beta\right)^\ell. ~~~
\end{eqnarray}
Let $m$ be a point in $M.$ Pick a Darboux neighborhood  $U$ of $m$ as in the proof of Lemma \ref{lemma5.3}. Using (\ref{P1}), (\ref{P2}) and (\ref{P3}), we  rewrite Equation (\ref{1}) as follows:
\begin{eqnarray}\label{2}
\alpha_t\wedge\left(d\alpha_t\right)^k\wedge\beta_t\wedge\left(d\beta_t\right)^\ell |_U=t^{k+\ell+1}\alpha_0(E_\alpha)~ \alpha\wedge\left(d\alpha\right)^k\wedge\beta\wedge\left(d\beta\right)^\ell |_U\nonumber\\
+t^{k+\ell+1} \beta_0(E_\beta)~ \alpha\wedge\left(d\alpha\right)^k\wedge\beta\wedge\left(d\beta\right)^\ell|_U \nonumber\\
 +t^{k+\ell+2}\alpha\wedge\left(d\alpha\right)^k\wedge\beta\wedge\left(d\beta\right)^\ell |_U.
\end{eqnarray} Use the compatibility conditions for $\alpha_0$ and $\beta_0$ to get

$$\alpha_t\wedge\left(d\alpha_t\right)^k\wedge\beta_t\wedge\left(d\beta_t\right)^\ell=t^{k+\ell+2}\alpha\wedge\left(d\alpha\right)^k\wedge\beta\wedge\left(d\beta\right)^\ell \ne 0, \quad  \forall~ t\ne 0.$$
Moreover, $(d\alpha_t)^{k+1}=0$ if and only if $(d\alpha)^{k+1}=0$ and $(d\beta_t)^{k+1}=0$ if and only if $(d\beta)^{k+1}=0$. 
Thus $\left(\alpha_t, \beta_t \right)$  is  a contact pair of type $(k, \ell)$ for all $t\ne 0.$  Its associated pair of Reeb vector fields is  $(E_{\alpha_t}, E_{\beta_t})=( { 1 \over t} E_\alpha,  { 1 \over t} E_\beta)$ because 
these are the unique vectors $E_{\alpha_t}$ and $ E_{\beta_t}$ that satisfy:
$$\alpha_t\left (E_{\alpha_t}\right)=1, \qquad \qquad   i_{E_{\alpha_t}} d\alpha_t = i_{E_{\alpha_t}} d\beta_t=0;$$
$$\beta_t \left (E_{\beta_t}\right)=1,  \quad {\rm and}  \qquad   i_{E_{\beta_t}} d\alpha_t = i_{E_{\beta_t}} d\beta_t=0.$$

\end{proof}
\noindent Now we will establish the  converse of Proposition \ref{prop1}. Precisely, we have:
\begin{prop}\label{inverse}
Let $M$ be a closed  and oriented even-dimensional manifold  and let $\mathcal{F}^1$ and $ \mathcal{F}^2$ be  two foliations in  $M$  defined  by two  linearly independent closed 1-forms $\alpha_0$ and  $\beta_0$, respectively. Consider the linear deformation $(\alpha_t, \beta_t)$ given by:  
 \begin{eqnarray} \label{linear deformation}
 \alpha_t=\alpha_0+t\alpha \quad {\rm and} \quad  \beta_t=\beta_0+t\beta, ~ \forall ~ t \in \mathbb R,
 \end{eqnarray}  
 \noindent where $\alpha$ and $\beta$ are two  1-forms on $M.$
If   $\left(\alpha_t, \beta_t\right)$  is  a contact pair of type $(k, \ell)$ whose  associated  Reeb vector fields $(E_{\alpha_t}, E_{\beta_t})$ has the form $( { 1 \over t} X,  { 1 \over t} Y)$  with $X, Y \in \mathfrak X(M)$, for $t > 0$ then  $\left(\alpha, \beta\right)$  is a  contact pair of type $(k, \ell)$ whose pair of  Reeb vector fields 
$(E_\alpha, E_\beta)$ satisfies  the compatibility condition: $ \alpha_0(E_\alpha)=\beta_0(E_\beta)=0$  and $ \alpha_0(E_\beta)=\beta_0(E_\alpha)=0.$
\end{prop}
 
\begin{proof}

 First, we fix an  orientation of $M$ by choosing a volume form $\Omega$  and then we consider the following smooth functions $A$, $B$ and $C$ given by:
\begin{eqnarray}\label{3.9}
{\left\{ \begin{array}{ccc}
A\Omega=\alpha\wedge\left(d\alpha\right)^k\wedge\beta\wedge\left(d\beta\right)^\ell; \\
B\Omega=\alpha_0\wedge\left(d\alpha\right)^k\wedge\beta\wedge\left(d\beta\right)^\ell+\alpha\wedge\left(d\alpha\right)^k\wedge\beta_0\wedge\left(d\beta\right)^\ell; \\
C\Omega=\alpha_0\wedge\left(d\alpha\right)^k\wedge\beta_0\wedge\left(d\beta\right)^\ell.
\end{array}\right.}
\end{eqnarray}
Consider the following quadratic function in the variable $t$:  
\begin{equation}
P(m, t)=t^2 A(m)+t B(m)+C (m), ~ ~  \forall ~m \in M, ~  \forall~ t.
\end{equation}
Equation (\ref{1}) becomes  \begin{equation}\label{poly} \alpha_t\wedge\left(d\alpha_t\right)^k\wedge\beta_t\wedge\left(d\beta_t\right)^\ell = t^{k+\ell} P(t) \Omega.\end{equation}
Assume $\left(\alpha_t, \beta_t\right)$  is  a contact pair of type $(k, \ell)$  for $t > 0 ,$  with  associated  Reeb vector fields $(E_{\alpha_t}, E_{\beta_t})=( { 1 \over t} X,  { 1 \over t} Y)$. We must show that that $\left(\alpha, \beta\right)$  is a  contact pair of type $(k, \ell)$ whose pair of  Reeb vector fields 
$(E_\alpha, E_\beta)=(X,Y)$ satisfies  $ \alpha_0(E_\alpha)=\beta_0(E_\beta)=0$  and $ \alpha_0(E_\beta)=\beta_0(E_\alpha)=0.$

\medskip 
Consider $t>0$ small enough. We can assume, without loss of generality, that $\alpha_t\wedge\left(d\alpha_t\right)^k\wedge\beta_t\wedge\left(d\beta_t\right)^\ell >0$  since the case
$\alpha_t\wedge\left(d\alpha_t\right)^k\wedge\beta_t\wedge\left(d\beta_t\right)^\ell <0$ can be proved by similar arguments.
\medskip

\noindent {\bf Step 1:}  First,  Lemma \ref{lemma5.4} says that $C$ is  identically zero and $$P(m, t)=t (t  A(m)+ B(m))$$ for  all $m \in M$.

\medskip 
\noindent {\bf Step 2:} We will prove that $B=0$ by using argument similar to those in Step 1. First, we remark that $B(m)\geq 0$ for any $m \in M$. Indeed if $B(m_0)<0$ for some $m_0$ then 
  $P(t, m)=t (t  A(m)+ B(m)) <0$  for all $t>0$ small enough and  points $m$ near $m_0$ because $P$ is smooth. So,  this would imply
 $$\alpha_t\wedge\left(d\alpha_t\right)^k\wedge\beta_t\wedge\left(d\beta_t\right)^\ell = t^{k+\ell} P(m, t)  \Omega <0$$  for $m$ near $m_0$ and for $t>0$ small enough. This is a contradiction. Thus $B \geq 0.$
 Now, apply Stokes' theorem  to get:
\begin{eqnarray*}
\int_M\alpha_0\wedge\left(d\alpha\right)^k\wedge\beta\wedge\left(d\beta\right)^\ell &=& - \int_{ M}d\left(\alpha_0\wedge \alpha\wedge \left(d\alpha\right)^{k-1}\wedge\beta\wedge\left(d\beta\right)^\ell\right)\\
&=&- \int_{ \partial M}\alpha_0\wedge \alpha\wedge \left(d\alpha\right)^{k-1}\wedge\beta\wedge\left(d\beta\right)^\ell\\
&=&0,
\end{eqnarray*}
since $M$ is closed. Similarly, we have
\begin{eqnarray*}
\int_M\alpha\wedge\left(d\alpha\right)^k\wedge\beta_0\wedge\left(d\beta\right)^\ell &=& (-1)^{2k+2} \int_{ M}d\left(\alpha\wedge   \left(d\alpha\right)^{k}\wedge\beta_0\wedge \beta \wedge \left(d\beta\right)^{\ell-1}\right)\\
&=&(-1)^{2k+2}  \int_{ \partial M}\alpha\wedge\left(d\alpha\right)^{\ell}\wedge\beta_0\wedge \beta\wedge \left(d\beta\right)^{\ell-1}\\
&=&0.
\end{eqnarray*}
Therefore, 
\begin{eqnarray*}
\int_M B\Omega=0.
\end{eqnarray*}
\noindent Consequently $B=0$ and  for $t>0$ small enough, we have $$P(m , t)=t^2A(m), ~\forall ~m \in M.$$ So,  we deduce:
\begin{eqnarray}
\alpha_t\wedge\left(d\alpha_t\right)^k\wedge\beta_t\wedge\left(d\beta_t\right)^\ell = t^{k+\ell+2} ~ \alpha\wedge\left(d\alpha\right)^k\wedge\beta\wedge\left(d\beta\right)^\ell >0.
\end{eqnarray}

\noindent Furthermore,  $(d\alpha_t)^{k+1}=0$ if and only if $(d\alpha)^{k+1}=0$ and $(d\beta_t)^{k+1}=0$ if and only if $(d\beta)^{k+1}=0$. Hence $\left(\alpha, \beta\right)$  is  a contact pair of type $(k, \ell).$\\
 To prove that its pair of Reeb vector fields  is $(E_\alpha, E_\beta)=(X,Y)$, we  use the following  relations that are satisfied by the pair $(E_{\alpha_t}, E_{\beta_t})$ of  Reeb vector fields of the contact pair $(\alpha_t, \beta_t)$:
\begin{equation} \label{a}   \alpha_t ( E_{\alpha_t})= \alpha_t \left( {1 \over t} X\right) ={1 \over t} \alpha_t(X)=1; \end{equation}
\begin{equation} \label{b}   \alpha_t ( E_{\beta_t})= \alpha_t \left( {1 \over t} Y\right) ={1 \over t} \alpha_t(Y)=0 ;\end{equation}
\begin{equation} \label{A}   \beta_t ( E_{\beta_t})= \beta_t \left( {1 \over t} Y\right) ={1 \over t} \beta_t(X)=1; \end{equation}
\begin{equation} \label{B}   \beta_t ( E_{\alpha_t})= \beta_t \left( {1 \over t} X\right) ={1 \over t} \beta_t(X)=0; \end{equation}
\begin{equation} \label{c}  i_{E_{\alpha_t}} d\alpha_t = i_{E_{\alpha_t}} d\beta_t=0 \quad{\rm and}\quad
i_{E_{\beta_t}} d\alpha_t = i_{E_{\beta_t}} d\beta_t=0.
 \end{equation}

\noindent From  Equation (\ref{a}) we get:
\begin{equation}\label{a'} \alpha_0(X) = 0\quad{\rm and}\quad  \alpha(X)=1. \end{equation} 

\noindent Equation (\ref{b}) implies
\begin{equation} \label{b'}  \alpha_0(Y) = 0\quad{\rm and}\quad  \alpha(Y)=0. \end{equation} 

\noindent Equations  (\ref{A}) and  (\ref{B}) imply
\begin{equation}\label{A'} \beta_0(Y) = 0; \quad  \beta(Y)=1, \qquad \beta_0(X) = 0; \quad{\rm and}\quad  \beta(X)=0 . \end{equation} 

\noindent While Equation (\ref{c})  gives
\begin{equation} \label{c'}    i_{X} d \alpha=   i_{X} d \beta=0 \quad{\rm and}\quad  i_{Y} d \alpha=   i_{Y} d \beta=0.  \end{equation} 

\noindent From (\ref{a'}),  (\ref{b'}),  (\ref{A'})  and (\ref{c'}) it follows that $X=E_{\alpha}$, $ \alpha_0(E_\alpha)=0$, and $\alpha_0(E_\beta)=0$ as well as $Y= E_{\beta}$, and $ \beta_0(E_\alpha)=  \beta_0(E_\beta)=0$. 

\noindent This completes the proof.

\end{proof}

 \section{Applications }
In this section, we give some applications of our main results.
\begin{prop}
Let $G_1$  (resp. $G_2$)  be  a  $K$-contact Lie group of dimension $2n+1$  (resp. $2m+1$) endowed  with  a foliation   $\mathcal{F}^1$  (resp.  $\mathcal{F}^2)$ of  codimension 1   defined by a harmonic  right invariant  1-form  on  $G_1$  (resp. $G_2$). Then the obvious pair of  product foliations induced by  $\mathcal{F}^1$ and  $\mathcal{F}^2$
on $G=G_1 \times G_2$ can be linearly deformed into a pair $\left(\mathcal{F}_t^1, \mathcal{F}_t^2\right)$  of transverse characteristic foliations in $G$.
\end{prop}

\begin{proof}
Let $\alpha_0^+$ (resp. $\beta_0^+$) be a closed  right invariant  1-form in $G_1$ (resp.$G_2$)  which is harmonic  with respect to a Riemannian  metric $h_1$ (resp.  $h_2$) defining the foliation  $\mathcal{F}^1$ (resp  $\mathcal{F}^2$).
By Theorem \ref{Ameth}, there exists a right invariant contact 1-form  $\alpha^+$ on $G_1$ such that  $(\alpha_t)^+=\alpha_0^++t\alpha^+$ is a  linear deformation. Similarly, there exists a right invariant contact $\beta^+$ on  $G_2$ such that $(\beta_t)^+=\beta_0^++t\beta^+$ is a linear deformation. Using Example \ref{Example C.P}, we see that $\left(\alpha^+, \beta^+\right)$ is a contact pair on  $G=G_1\times G_2$.\\
Theorem \ref{thm1} shows that  $\left((\alpha_t)^+, (\beta_t)^+\right)$ is a contact pair on  $G$. Thus,  it defines a pair of transverse and complementary characteristic foliations $\left(\mathcal{F}_t^1, \mathcal{F}_t^2\right)$  of $(\alpha_t,  \beta_t)$.
\end{proof}

 \begin{prop}
 Let $M$ and $N$ be two closed contact manifolds of dimensions $2h+1$ and $2k+1$ respectively. Then $M\times N\times\mathbb{T}^2$ admit a contact pair which is invariant under the action of $\mathbb{T}^2$, and such that  $M\times N\times\{p\}$ is a pair contact manifold  for all  $p\in\mathbb{T}^2$.
 \end{prop}
 \begin{proof}
 Let $(M, \alpha)$ and $(N, \beta)$ be two closed contact manifolds. Using the theorem \ref{Theorem Bourgeois}, we see that $N\times\mathbb{T}^2$ admits a contact structure which is  invariant under the action of $\mathbb{T}^2$.
 Using the example \ref{Example C.P} we obtain that  $(\alpha, \beta)$ is a  contact pair on the product  $M\times N\times\mathbb{T}^2$. If  $M$ has the trivial action then $M\times N\times\mathbb{T}^2$ has the trivial diagonal action, we deduce that  $(\alpha, \beta)$ is invariant under the action of $\mathbb{T}^2$.
  \end{proof}
 

\noindent {\bf Conflict of interest statement:}  On behalf of all authors, the corresponding author states that there is no conflict of interest. \\

\noindent {\bf Data Availability Statement:} The authors declare that all data supporting the findings of this study are available within the article.

\end{document}